\documentclass[12pt,a4paper]{article}
\usepackage{amsfonts}
\usepackage{amssymb}
\usepackage{mathrsfs}
\usepackage{amsmath}
\usepackage{amsthm}
\usepackage{enumerate}

\allowdisplaybreaks
\newtheorem{defn}{Definition}[section]
\newtheorem{thm}[defn]{Theorem}

\newtheorem{prop}[defn]{Proposition}
\newtheorem{cor}[defn]{Corollary}

\newtheorem{re}[defn]{Remark}
\numberwithin{equation}{section}
\begin{document}
\title{{\bf $3$-BiHom-Lie superalgebras induced by BiHom-Lie superalgebras}}
\author{\normalsize \bf  Abdelkader Ben Hassine\small{$^{1,2}$},  Sami Mabrouk\small{$^{3}$}, Othmen Ncib\small{$^{4}$}}
\date{{\small{$^{1}$ Department of Mathematics, Faculty of Science and Arts at
Belqarn, P. O. Box 60, Sabt Al-Alaya 61985, University of Bisha, Saudi Arabia \\  \small{$^{2}$    Faculty of Sciences, University of Sfax,   BP
1171, 3000 Sfax, Tunisia \\  \small{$^{3,4}$} Faculty of Sciences, University of Gafsa,   BP
2100, Gafsa, Tunisia
 }}}} \maketitle

{\bf\begin{center}{Abstract}\end{center}}
 The purpose of this paper is to study the relationships between
a BiHom-Lie superalgebras and its induced 3-BiHom-Lie superalgebras. We introduce the notion of $(\alpha^s,\beta^r)$-derivation, $(\alpha^s,\beta^r)$-quasiderivation and generalized $(\alpha^s,\beta^r)$-derivation of 3-BiHom-Lie superalgebras, and their relation with derivation of BiHom-Lie superalgebras. We introduce also the concepts of Rota-Baxter operators and Nijenhuis Operators of BiHom $3$-Lie superalgebras. We also explore the construction of $3$-BiHom-Lie superalgebras by using Rota-Baxter of BiHom-Lie superalgebras.


\noindent\textbf{Keywords:}

\renewcommand{\thefootnote}{\fnsymbol{footnote}}
\footnote[0]{ Corresponding author(S.Mabrouk): mabrouksami00@yahoo.fr}

%----------------------------------------------------------------------------------------------------------------------
 \section*{Introduction}
BiHom-type generalizations of $n$-ary Nambu-Lie algebras, called $n$-ary BiHom-Nambu-Lie algebras, were introduced by
 Kitouni, Makhlouf, and  Silvestrov  in \cite{Kitouni&Makhlouf&Silvestrov}. Each $n$-ary BiHom-Nambu-Lie algebra has $(n-1)$-linear twisting maps, which appear
in a twisted generalization of the $n$-ary Nambu identity called the $n$-ary BiHom-Nambu identity. If the twisting maps are all
equal to the identity, one recovers an $n$-ary Nambu-Lie algebra. The twisting maps provide a substantial amount of freedom
in manipulating Nambu-Lie algebras.
In recent years, Rota-Baxter (associative) algebras, originated from the work of G. Baxter \cite{Baxter} in probability and populated
by the work of Cartier and Rota \cite{Cartier,Rota}, have also been studied in connection with many areas of mathematics and physics,
including combinatorics, number theory, operators and quantum field theory \cite{Rota1}. Furthermore, Rota-Baxter operators on a Lie algebra are an operator form of the classical
Yang-Baxter equations and contribute to the study of integrable systems \cite{Bai&Guo&Ni}. Further Rota-Baxter $3$-Lie algebras are
closely related to pre-Lie algebras \cite{Bai&Guo&Li&Wu}. Rota-Baxter of multiplicative $3$-ary Hom-Nambu-Lie algebras were introduced by Sun and Chen, in \cite{Sun&Chen}.

Deformations of $n$-Lie algebras have been
studied from several aspects. See  \cite{Azcárraga,Figueroa}
for more details. In particular, a notion of a Nijenhuis
operator on a $3$-Lie algebra was introduced in  \cite{Zhang}
in the study of the $1$-order deformations of a $3$-Lie algebra. But there are some quite strong conditions in this
definition of a Nijenhuis operator. In the case of Lie algebras, one could obtain fruitful results by considering
one-parameter infinitesimal deformations, i.e. $1$-order deformations. However, for $n$-Lie algebras, we believe that
one should consider $(n-1)$-order deformations to obtain
similar results. In \cite{Figueroa}, for $3$-Lie algebras, the author had already considered $2$-order deformations. For the case of Hom-Lie superalgebras, the authors in \cite{Liu&chen&Ma} give the notion of Hom-Nijenhuis operator.

Thus it is time to study $3$-BiHom-Lie superalgebras, Rota-Baxter algebras and Nijenhuis operator together to get a suitable
definition of Rota-Baxter of $3$-BiHom-Lie superalgebras induced by BiHom-Lie superalgebras. Similarly, we give the relationship between Nijenhuis operator of $3$-BiHom-Lie superalgebras and BiHom-Lie superalgebras.

This paper is organized as follows: In Section $1$, we recall the concepts of BiHom-Lie superalgebras and introduce the notion of $3$-BiHom-Lie superalgebras. The construction of $3$-BiHom-Lie superalgebras induced by BiHom-Lie superalgebras are established in Section $2$. In section $3$, we give the definition of $(\alpha^s,\beta^r)$-derivation and $(\alpha^s,\beta^r)$-quasiderivation of $3$-BiHom-Lie superalgebras. In section $4$, we give the definition of Rota-Baxter of $3$-BiHom-Lie superalgebras and the realizations of Rota-Baxter  of $3$-BiHom-Lie superalgebras from Rota-Baxter  BiHom-Lie superalgebras.
The Section $5$ is dedicated to study the second order deformation of $3$-BiHom-Lie superalgebras, and introduce the notion of Nijenhuis operator on $3$-BiHom-Lie superalgebras, which could generate a trivial deformation. In the other part of this section we give some properties and results of Nijenhuis operators.

\section{ Definitions and Notations}
In this section, we review basics definition of BiHom-Lie superalgebras, $3$-Lie superalgebras and generalize the notion of $3$-BiHom-Lie algebras to the super case.

Let $V= V_{\overline 0}\oplus V_{\overline 1}$ be a $\mathbb{Z}_2$-graded vector space. If
$v \in V$ is a homogenous element, then its degree will be denoted by $|v|$,
where $|v|\in \mathbb{Z}_2$ and $\mathbb{Z}_2=\{\overline 0,\overline 1\}$. Let $End(V)$ be the $\mathbb{Z}_2$-graded vector space of
endomorphisms of a $\mathbb{Z}_2$-graded vector space $V= V_{\overline 0}\oplus V_{\overline 1}$, we denote by $\mathcal{H}(V)$ the set of homogenous elements of $V$. The composition of
two endomorphisms $a\circ b$ determines the structure of superalgebra in $End(V)$,
and the graded binary commutator $[a, b] = a\circ b - (-1)^{|a||b|}b \circ a$ induces the
structure of Lie superalgebras in $End(V)$. %The supertrace of an endomorphism
%$\alpha : V \rightarrow V$ can be defined by
%$$
%Str(\alpha) =\left\{
%          \begin{array}{ll}
%           Tr(\alpha|V_{\overline0} ) - Tr(\alpha|V_{\overline 1} ), & \hbox{$ \text{if~ $\alpha$ ~is~ even}$;} \\
%            0, & \hbox{$ \text{if\  $\alpha$\  is\ odd}$.}
%          \end{array}
%        \right.
%$$
%For any endomorphisms $\sigma$, $\tau$ it holds $Str([\sigma, \tau ])=0$.

\begin{defn}\cite{Guo}
A BiHom-Lie superalgebra is a triple $(\mathfrak g, [\cdot, \cdot], \alpha,\beta)$ consisting
of a $\mathbb{Z}_2$-graded vector space $\mathfrak g=\mathfrak g_{\overline 0}\oplus \mathfrak g_{\overline 1}$, an even bilinear map $[\cdot,\cdot] : \mathfrak g\times\mathfrak g\longrightarrow\mathfrak g$ and a two even homomorphisms $\alpha,\beta: \mathfrak g \rightarrow \mathfrak g$ satisfying the following
 identities:
\begin{eqnarray}
& &\alpha\circ\beta=\beta\circ\alpha,\label{Multiplicativity1}\\
& &\alpha([x,y])=[\alpha(x),\alpha(y)],\ \beta([x,y])=[\beta(x),\beta(y)],\label{Multiplicativity2}\\
& &[\beta(x),\alpha(y)] = -(-1)^{|x||y|}[\beta(y), \alpha(x)],\ \label{BiHom-skewsupersymmetry}\\
& &\displaystyle\circlearrowleft_{x,y,z}(-1)^{|x||z|}[\beta^2(x), [\beta(y),\alpha(z)]] = 0\ \label{BiHomSuperJacobi} .
\end{eqnarray}
where $x, y$ and $z$ are homogeneous elements in $\mathfrak g$. The condition \eqref{BiHomSuperJacobi} is called BiHom-super-Jacobi identity.\\If the conditions \eqref{Multiplicativity1} and \eqref{Multiplicativity2}  are not satisfying, then BiHom-Lie superalgebra is called nonmultiplicative BiHom-Lie superalgebra.
\end{defn}
\begin{defn} \cite{Cantarini&Kac}
A $\mathbb{Z}_2$-graded vector space $\mathfrak{g}=\mathfrak g_{\overline 0}\oplus\mathfrak g_{\overline 1}$ is said to be a $3$-Lie superalgebra, if it is endowed with a trilinear map (bracket) $[\cdot,\cdot,\cdot]:\mathfrak g\times \mathfrak g \times \mathfrak g\rightarrow \mathfrak g$. If it satisfies the following conditions:
{\small\begin{align*}
[x, y, z]& = -(-1)^{|x||y|}[y, x, z],\\ [x, y, z] &= -(-1)^{|y||z|}[x, z, y],\\
[x, y, [z, u, v]] = [[x, y, z], u, v]&+(-1)^{|z|(|x|+|y|)}[z, [x, y, u], v]+(-1)^{(|z|+|u|)(|x|+|y|)}[z, u, [x, y, v]],
\end{align*}}
where $x, y, z, u, v \in\mathfrak g$ are homogeneous elements.

\end{defn}

\begin{defn}
A nonmultiplicative $3$-BiHom-Lie superalgebra is a quadruple $(\mathfrak g, [\cdot, \cdot,\cdot],\alpha ,\beta)$
consisting of a $\mathbb{Z}_2$-graded vector space $\mathfrak g =\mathfrak g_{\overline 0}\oplus\mathfrak g_{\overline 1}$, an even trilinear map
(bracket) $[\cdot,\cdot,\cdot]:\mathfrak g \times \mathfrak  g \times \mathfrak  g \rightarrow\mathfrak  g$ and two even  endomorphisms $\alpha,\beta:\mathfrak g \rightarrow \mathfrak g$. If it satisfies the following conditions:
\begin{eqnarray}
&&[\beta(x),\beta(y), \alpha(z)]=-(-1)^{|x||y|}[\beta(y),\beta(x),\alpha( z)],\nonumber\\
&&~~  [\beta(x), \beta(y),\alpha(z)] = -(-1)^{|y||y|}[\beta(x),\beta(z),\alpha(y)];\nonumber\\
&&[\beta^2(x), \beta^2(y), [\beta(z),\beta(u),\alpha(v)]]=(-1)^{(|u|+|v|)(|x|+|y|+|z|)}[ \beta^2(u),\beta^2(v),[\beta(x),\beta(y),\alpha(z)]]\nonumber\\
&&- (-1)^{(|z|+|v|)(|x|+|y|)+|u||v|}
[\beta^2(z),\beta^2(v),[\beta(x),\beta(y),\alpha(u)]]\nonumber\\
&&+(-1)^{(|z|+|u|)(|x|+|y|)}[\beta^2(z),\beta^2(u), [\beta(x),\beta(y),\alpha(v)]],\label{BiHom-Nambu}
\end{eqnarray}
where $x, y, z, u, v \in\mathfrak g$ are homogeneous elements.
The condition \eqref{BiHom-Nambu} is called the $3$-BiHom-super-Jacobi identity.
\end{defn}
\begin{re}
The identity \eqref{BiHom-Nambu} is equivalent to
{\small\begin{eqnarray*}
&[\beta^2(x), \beta^2(y), [\beta(z),\beta(u),\alpha(v)]]=(-1)^{|z||v|}\displaystyle\circlearrowleft_{uvz}(-1)^{\gamma}[ \beta^2(u),\beta^2(v),[\beta(x),\beta(y),\alpha(z)]],
\end{eqnarray*}}
where $\gamma=(-1)^{(|u|+|v|)(|x|+|y|)+|z||u|}.$
\end{re}
\begin{defn}
A $3$-BiHom-Lie superalgebra is a  nonmultiplicative $3$-BiHom-Lie superalgebra $(\mathfrak g, [\cdot, \cdot,\cdot],\alpha ,\beta)$ such that
\begin{eqnarray}
&&\alpha\circ\beta=\beta\circ\alpha,\\
&&\alpha([x,y,z])=[\alpha(x),\alpha(y),\alpha(z)]\;\;and\;\;\beta([x,y,z])=[\beta(x),\beta(y),\beta(z)],
\end{eqnarray}
for all $x,y,z\in\mathcal{H}(\mathfrak{g}).$
\end{defn}
\begin{re}
If \; $\alpha=\beta=Id$ we recover the $3$-Lie superalgebra.
\end{re}
%
%A representation of the BiHom-Lie superalgebra $(\mathfrak g, [~,~], \alpha,\beta)$
%on a $\mathbb{Z}_2$-graded vector space $V=V_{\overline 0}\oplus  V_{\overline 1}$ with respect to $A, B \in \mathfrak{gl}(V )_{\overline 0}$ is an
%even linear map $\rho: \mathfrak g \rightarrow \mathfrak{gl}(V)$, such that for any homogeneous elements
%$x, y\in\mathfrak  g$, the following equalities are satisfied:
%\begin{eqnarray*}
%&\rho(\alpha(x))\circ A=A\circ\rho(x),&\\
%&\rho(\beta(x))\circ B=B\circ\rho(x),&\\
%&\rho([\beta(x),y])\circ B=\rho(\alpha\beta(x))\circ\rho(y)-(-1)^{(|x||y|)}\rho(\beta(y))\circ\rho(\alpha(x)).&
%\end{eqnarray*}
%\end{defn}
\begin{prop}
Let $(\mathfrak{g},[\cdot,\cdot,\cdot])$ be a $3$-Lie superalgebra and $\alpha,\beta:\mathfrak{g}\rightarrow\mathfrak{g}$ are two commuting morphisms on $\mathfrak{g}$, then
$(\mathfrak{g},[\cdot,\cdot,\cdot]_{\alpha,\beta},\alpha,\beta)$ is a $3$-BiHom-Lie superalgebra,\\
where $[x,y,z]_{\alpha,\beta}=[\alpha(x),\alpha(y),\beta(z)]$.

\end{prop}
\begin{proof}~

It is easy to see that, $\forall x,y,z\in\mathcal{H(\mathfrak{g})}$ we are :\\

$[\beta(x),\beta(y),\alpha(z)]_{\alpha,\beta}=-(-1)^{|x||y|}[\beta(y),\beta(x),\alpha(z)]_{\alpha,\beta}$\\

 and
$[\beta(x),\beta(y),\alpha(z)]_{\alpha,\beta}=-(-1)^{|y||z|}[\beta(x),\beta(z),\alpha(y)]_{\alpha,\beta}$.\\

Let $x,y,z,t,u\in\mathcal{H(\mathfrak{g})}$,we have\\
\begin{eqnarray*}
[\beta^2(x),\beta^2(y),[\beta(z),\beta(t),\alpha(u)]_{\alpha,\beta}]_{\alpha,\beta}&=&
[\alpha\beta^2(x),\alpha\beta^2(y),[\beta^2(z),\beta^2(t),\alpha\beta(u)]_{\alpha,\beta}]\\
&=&[\alpha\beta^2(x),\alpha\beta^2(y),[\alpha\beta^2(z),\alpha\beta^2(t),\alpha\beta^2(u)]]\\
&=&\alpha\beta^2\Big([x,y,[z,t,u]]\Big),
\end{eqnarray*}
then the $3$-BiHom-super-Jacobi identity it satisfies.
\end{proof}
\section{ $3$-BiHom-Lie Superalgebras Induced by BiHom-Lie
Superalgebras}
Now we generalize the result given in \cite{Guan&Chen&Sun} to super case. Given a nonmultplicative BiHom-Lie superalgebra
$(\mathfrak g, [\cdot,\cdot],\alpha,\beta)$ and a linear form $\tau:\mathfrak g\rightarrow K$. %and $\rho: \mathfrak g \rightarrow \mathfrak{gl}(V)$ is a representation of $(\mathfrak g, [~, ~], \alpha,\beta)$.
 For any $x_1, x_2,x_3 \in\mathcal{H}(\mathfrak g)$, we define
the 3-ary bracket by
\begin{equation}\label{crochet_n}
[x_1,x_2,x_3]_\tau=\tau(x_1)[x_2, x_3] - (-1)^{|x_1||x_2|}\tau(x_2)[x_1, x_3]+(-1)^{|x_3|(|x_1|+|x_2|)}\tau(x_3)[x_1, x_2].
\end{equation}

\begin{thm}

Let $(\mathfrak g, [\cdot, \cdot], \alpha,\beta)$ be a nonmultiplicative BiHom-Lie superalgebra. % and
%$\rho: \mathfrak g \rightarrow \mathfrak{gl}(V)$ is a representation of $(\mathfrak g, [~, ~], \alpha,\beta)$ satisfying $str\rho\circ \alpha = str\rho$ and
%$str\rho\circ \beta = str\rho$, then
 If the conditions
\begin{eqnarray}
  \tau([x,y])=0 &\text{and}&\tau(x)\tau(\beta(y)) =\tau(y)\tau(\beta(x)), \label{ConditionTau1}\\
  \tau(\alpha(x))\beta(y)&=&\tau(\beta(x))\alpha(y),\label{ConditionTau2}
   \end{eqnarray}
are satisfied for any $x,y\in\mathfrak{g}$, then $(\mathfrak g, [\cdot, \cdot, \cdot]_\tau, \alpha,\beta)$ is a nonmultiplicative 3-BiHom-Lie superalgebra.\\
We say that $(\mathfrak g, [\cdot,\cdot,\cdot]_\tau, \alpha, \beta)$ induced by $(\mathfrak g, [\cdot, \cdot], \alpha,\beta)$, it is denoted by $\mathfrak g_\tau$.

\end{thm}

\begin{proof}

For any $x_1, x_2, x_3 \in \mathcal{H}(\mathfrak g)$, we have
{\small\begin{eqnarray*}
[\beta(x_1),\beta(x_2),\alpha(x_3)]_\tau&=&\tau(\beta(x_1))[\beta(x_1), \alpha(x_3)] - (-1)^{|x_1||x_2|}\tau(\beta(x_2))[\beta(x_1),\alpha(x_3)]\\
&+&(-1)^{|x_3|(|x_1|+|x_2|)}\tau(\alpha(x_3))[\beta(x_1), \beta(x_2)]\\
&=&-(-1)^{|x_1||x_2|}\Big(\tau(\beta(x_2))[\beta(x_1),\alpha(x_3)]- (-1)^{|x_1||x_2|}\tau(\beta(x_1))[\beta(x_2),\alpha(x_3)]\\
&-&(-1)^{|x_3|(|x_1|+|x_2|)+|x_1||x_2|}\tau(\alpha(x_3))[\beta(x_1), \beta(x_2)]\Big)\\
&=&-(-1)^{|x_1||x_2|}\Big(\tau(\beta(x_2))[\beta(x_1),\alpha(x_3)]- (-1)^{|x_1||x_2|}\tau(\beta(x_1))[\beta(x_2),\alpha(x_3)]\\
&-&(-1)^{|x_3|(|x_1|+|x_2|)}\tau(\alpha(x_3))[\beta(x_2), \beta(x_1)]\Big)\\
&=&-(-1)^{|x_1||x_2|}[\beta(x_2),\beta(x_1),\alpha(x_3)]_\tau.
\end{eqnarray*}}
Similarly, one gets $[\beta(x_1),\beta(x_2),\alpha(x_3)]_\tau=-(-1)^{|x_2||x_3|}[\beta(x_1),\beta(x_3),\alpha(x_2)]_\tau$.\\

%We know that $[~,~,~]_\tau$ is an trilinear skew-supersymmetric map.

We next show that $[\cdot,\cdot,\cdot]_{\tau}$ is satisfies the $3$-BiHom-super-Jacobi identity.\\
Let $L$ be its left-hand side, and $R$ its right-hand side of the identity \eqref{BiHom-Nambu}. Using the definition of $[\cdot,\cdot,\cdot]_{\tau}$, the equations \eqref{ConditionTau1} and \eqref{ConditionTau2}, we have
{\small\begin{align*}
L=&[\beta^2(x_1),\beta^2(x_2),[\beta(y_1),\beta(y_2),\alpha(y_3)]_\tau]_\tau\\
=&\tau(\beta(y_1))[\beta^2(x_1),\beta^2(x_2),[\beta(y_2),\alpha(y_3)]]_\tau
-(-1)^{|y_1||y_2|}\tau(\beta(y_2))[\beta^2(x_1),\beta^2(x_2),[\beta(y_1),\alpha(y_3)]]_\tau\\
&+(-1)^{|y_3|(|y_1|+|y_2|)}\tau(\beta(y_3))[\beta^2(x_1),\beta^2(x_2),[\beta(y_1),\alpha(y_2)]]_\tau\\
=&\tau(\beta(y_1))\{\tau(\beta^2(x_1))[\beta^2(x_2),[\beta(y_2),\alpha(y_3)]]
-(-1)^{|x_1||x_2|}\tau(\beta^2(x_2))[\beta^2(x_1),[\beta(y_2),\alpha(y_3)]]\}\\
&-(-1)^{|y_1||y_2|}\tau(\beta(y_2))\{\tau(\beta^2(x_1))[\beta^2(x_2),[\beta(y_1),\alpha(y_3)]]\\
&-(-1)^{|x_1||x_2|}\tau(\beta^2(x_2))[\beta^2(x_1),[\beta(y_1),\alpha(y_3)]]\}\\
&+(-1)^{|y_3|(|y_1|+|y_2|)}\tau(\beta(y_3))\{\tau(\beta^2(x_1))[\beta^2(x_2),[\beta(y_1),\alpha(y_2)]]\\
&-(-1)^{|x_1||x_2|}\tau(\beta^2(x_2))[\beta^2(x_1),[\beta(y_1),\alpha(y_2)]]\}\\
=&\tau(\beta(y_1))\tau(\beta^2(x_1))[\beta^2(x_2),[\beta(y_2),\alpha(y_3)]]
-(-1)^{|x_1||x_2|}\tau(\beta(y_1))\tau(\beta^2(x_2))[\beta^2(x_1),[\beta(y_2),\alpha(y_3)]]\\
&-(-1)^{|y_1||y_2|}\tau(\beta(y_2))\tau(\beta^2(x_1))[\beta^2(x_2),[\beta(y_1),\alpha(y_3)]]\\
&+(-1)^{|x_1||x_2|+|y_1||y_2|}\tau(\beta(y_2))\tau(\beta^2(x_2))[\beta^2(x_1),[\beta(y_1),\alpha(y_3)]]\\
&+(-1)^{|y_3|(|y_1|+|y_2|)}\tau(\beta(y_3))\tau(\beta^2(x_1))[\beta^2(x_2),[\beta(y_1),\alpha(y_2)]]\\
&-(-1)^{|x_1||x_2|}(-1)^{|y_3|(|y_1|+|y_2|)}\tau(\beta(y_3))\tau(\beta^2(x_2))[\beta^2(x_1),[\beta(y_1),\alpha(y_2)]]\\
\end{align*}}
And
{\small\begin{align*}
R=&(-1)^{(|x_1|+|x_2|+|y_1|)(|y_2|+|y_3|)}[\beta^2(y_2),\beta^2(y_3),[\beta(x_1),\beta(x_2),\alpha(y_1)]_{\tau}]_{\tau}\\
&-(-1)^{(|x_1|+|x_2|)(|y_1|+|y_3|)}(-1)^{|y_2||y_3|}[\beta^2(y_1),\beta^2(y_3),[\beta(x_1),\beta(x_2),\alpha(y_2)]_{\tau}]_{\tau}\\
&+(-1)^{(|x_1|+|x_2|)(|y_1|+|y_2|)}[\beta^2(y_1),\beta^2(y_2),[\beta(x_1),\beta(x_2),\alpha(y_3)]_{\tau}]_{\tau}\\
=&(-1)^{(|x_1|+|x_2|+|y_1|)(|y_2|+|y_3|)}\tau(\beta(x_1))\tau(\beta^2(y_2))[\beta^2(y_3),[\beta(x_2),\alpha(y_1)]]\\
&-(-1)^{(|x_1|+|x_2|+|y_1|)(|y_2|+|y_3|)}(-1)^{|y_2||y_3|}\tau(\beta(x_1))\tau(\beta^2(y_3))[\beta^2(y_2),[\beta(x_2),\alpha(y_1)]]\\
&-(-1)^{(|x_1|+|x_2|+|y_1|)(|y_2|+|y_3|)}(-1)^{|x_1||x_2|}\tau(\beta(x_2))\tau(\beta^2(y_2))[\beta^2(y_3),[\beta(x_1),\alpha(y_1)]]\\
&+(-1)^{(|x_1|+|x_2|+|y_1|)(|y_2|+|y_3|)}(-1)^{|x_1||x_2|+|y_2||y_3|}\tau(\beta(x_2))\tau(\beta^2(y_3))[\beta^2(y_2),[\beta(x_1),\alpha(y_1)]]\\
&+(-1)^{(|x_1|+|x_2|+|y_1|)(|y_2|+|y_3|)}(-1)^{|y_1|(|x_1|+|x_2|)}\tau(\beta(y_1))\tau(\beta^2(y_2))[\beta^2(y_3),[\beta(x_1),\alpha(x_2)]]\\
&-(-1)^{(|x_1|+|x_2|+|y_1|)(|y_2|+|y_3|)}(-1)^{|y_1|(|x_1|+|x_2|)+|y_2||y_3|}\tau(\beta(y_1))\tau(\beta^2(y_3))[\beta^2(y_2),[\beta(x_1),\alpha(x_2)]]\\
&-(-1)^{(|x_1|+|x_2|)(|y_1|+|y_3|)}(-1)^{|y_2||y_3|}\tau(\beta(x_1))\tau(\beta^2(y_1))[\beta^2(y_3),[\beta(x_2),\alpha(y_2)]]\\
&+(-1)^{(|x_1|+|x_2|)(|y_1|+|y_3|)}(-1)^{|y_3|(|y_1|+|y_2|)}\tau(\beta(x_1))\tau(\beta^2(y_3))[\beta^2(y_1),[\beta(x_2),\alpha(y_2)]]\\
&+(-1)^{(|x_1|+|x_2|)(|y_1|+|y_3|)}(-1)^{|y_2||y_3|+|x_1||x_2|}\tau(\beta(x_2))\tau(\beta^2(y_1))[\beta^2(y_3),[\beta(x_1),\alpha(y_2)]]\\
&-(-1)^{(|x_1|+|x_2|)(|y_1|+|y_3|)}(-1)^{|y_3|(|y_1|+|y_2|)+|x_1||x_2|}\tau(\beta(x_2))\tau(\beta^2(y_3))[\beta^2(y_1),[\beta(x_1),\alpha(y_2)]]\\
&-(-1)^{(|x_1|+|x_2|)(|y_1|+|y_2|+|y_3|)}(-1)^{|y_3||y_2|}\tau(\beta(y_2))\tau(\beta^2(y_1))[\beta^2(y_3),[\beta(x_1),\alpha(x_2)]]\\
&+(-1)^{(|x_1|+|x_2|)(|y_1|+|y_2|+|y_3|)}(-1)^{|y_3|(|y_1|+|y_2|)}\tau(\beta(y_2))\tau(\beta^2(y_3))[\beta^2(y_1),[\beta(x_1),\alpha(x_2)]]\\
&+(-1)^{(|x_1|+|x_2|)(|y_1|+|y_2|)}\tau(\beta(x_1))\tau(\beta^2(y_1))[\beta^2(y_2),[\beta(x_2),\alpha(y_3)]]\\
&-(-1)^{(|x_1|+|x_2|)(|y_1|+|y_2|)+|y_1||y_2|}\tau(\beta(x_1))\tau(\beta^2(y_2))[\beta^2(y_1),[\beta(x_2),\alpha(y_3)]]\\
&-(-1)^{(|x_1|+|x_2|)(|y_1|+|y_2|)+|x_1||x_2|}\tau(\beta(x_2))\tau(\beta^2(y_1))[\beta^2(y_2),[\beta(x_1),\alpha(y_3)]\\
&+(-1)^{(|x_1|+|x_2|)(|y_1|+|y_2|)}(-1)^ {(|x_1||x_2|+|y_1||y_2|)}\tau(\beta(x_2))\tau(\beta^2(y_2))[\beta^2(y_1),[\beta(x_1),\alpha(y_3)]]\\
&+(-1)^{(|x_1|+|x_2|)(|y_1|+|y_2|+|y_3|)}\tau(\beta(y_3))\tau(\beta^2(y_1))[\beta^2(y_2),[\beta(x_1),\alpha(x_2)]]\\
&-(-1)^{(|x_1|+|x_2|)(|y_1|+|y_2|+|y_3|)}(-1)^ {|y_1||y_2|}\tau(\beta(y_3))\tau(\beta^2(y_2))[\beta^2(y_1),[\beta(x_1),\alpha(x_2)]].
\end{align*}}
By the identity of BiHom-super-Jacobi and the equation \eqref{ConditionTau1}   we get \\
{\small\begin{align*}
L-R=&\tau(\beta(y_1))\tau(\beta^2(x_1))\Omega_1+\tau(\beta(y_1))\tau(\beta^2(x_2))\Omega_2+\tau(\beta(y_2))\tau(\beta^2(x_1))\Omega_3\\
&+\tau(\beta(y_2))\tau(\beta^2(x_2))\Omega_4+\tau(\beta(y_3))\tau(\beta^2(x_1))\Omega_5+\tau(\beta(y_3))\tau(\beta^2(x_2))\Omega_6.
\end{align*}}
Where
{\small\begin{align*}
\Omega_1=&[\beta^2(x_2),[\beta(y_2),\alpha(y_3)]]+(-1)^{(|x_1|+|x_2|)(|y_1|+|y_3|)}(-1)^{|y_2||y_3|}[\beta^2(y_3),[\beta(x_2),\alpha(y_2)]]\\
&-(-1)^{(|x_1|+|x_2|)(|y_1|+|y_2|)}(-1)^{|y_1||y_2|}[\beta^2(y_2),[\beta(x_2),\alpha(y_3)]]\\
\Omega_2=&-(-1)^{|x_1||x_2|}[\beta^2(x_1),[\beta(y_2),\alpha(y_3)]]\\
&-(-1)^{(|x_1|+|x_2|)(|y_1|+|y_3|)}(-1)^{|x_1||x_2|+|y_2||y_3|}[\beta^2(y_3),[\beta(x_1),\alpha(y_2)]]\\
&+(-1)^{(|x_1|+|x_2|)(|y_1|+|y_2|)}(-1)^{|x_1||x_2|}[\beta^2(y_2),[\beta(x_1),\alpha(y_3)]]\\
\Omega_3=&-(-1)^{|y_1||y_2|}[\beta^2(x_2),[\beta(y_1),\alpha(y_3)]]-(-1)^{(|x_1|+|x_2|+|y_1|)(|y_2|+|y_3|)}[\beta^2(y_3),[\beta(x_2),\alpha(y_1)]]\\
&+(-1)^{(|x_1|+|x_2|)(|y_1|+|y_2|)}(-1)^{|y_1||y_2|}[\beta^2(y_1),[\beta(x_2),\alpha(y_3)]]\\
\Omega_4=&(-1)^{|x_1||x_2|}(-1)^{|y_1||y_2|}[\beta^2(x_1),[\beta(y_1),\alpha(y_3)]]\\
&+(-1)^{(|x_1|+|x_2|+|y_1|)(|y_2|+|y_3|)}(-1)^{|x_1||x_2|}[\beta^2(y_3),[\beta(x_1),\alpha(y_1)]]\\
&-(-1)^{(|x_1|+|x_2|)(|y_1|+|y_2|)}(-1)^{|x_1||x_2|+|y_1||y_2|}[\beta^2(y_1),[\beta(x_1),\alpha(y_3)]]\\
\Omega_5=&(-1)^{|y_3|(|y_1|+|y_2|)}[\beta^2(x_2),[\beta(y_1),\alpha(y_2)]]\\
&+(-1)^{(|x_1|+|x_2|+|y_1|)(|y_2|+|y_3|)}(-1)^{|y_2||y_3|}[\beta^2(y_2),[\beta(x_2),\alpha(y_1)]]\\
&-(-1)^{(|x_1|+|x_2|)(|y_1|+|y_3|)}(-1)^{|y_3|(|y_1|+|y_2|)}[\beta^2(y_1),[\beta(x_2),\alpha(y_1)]]\\
\Omega_6=&-(-1)^{|y_3|(|y_1|+|y_2|)}(-1)^ {|x_1||x_2|}[\beta^2(x_1),[\beta(y_1),\alpha(y_2)]]\\
&-(-1)^{(|x_1|+|x_2|+|y_1|)(|y_2|+|y_3|)}(-1)^{|x_1||x_2|+|y_1||y_2|}[\beta^2(y_2),[\beta(x_1),\alpha(y_1)]]\\
&+(-1)^{(|x_1|+|x_2|)(|y_1|+|y_3|)}(-1)^{|y_3|(|y_1|+|y_2|)+|x_1||x_2|}[\beta^2(y_1),[\beta(x_1),\alpha(y_2)]].\\
\end{align*}}
If $x_1$ or $y_1$ is odd, then $\tau(\beta(y_1))\tau(\beta^2(x_1))=0$ which gives $\tau(\beta(y_1))\tau(\beta^2(x_1))\Omega_1=0$.\\
If $x_1$ and $y_1$ are even, then:\\
\begin{eqnarray*}
\Omega_1&=&[\beta^2(x_2),[\beta(y_2),\alpha(y_3)]]+(-1)^{|x_2||y_2|+|y_2||y_3|}[\beta^2(y_3),[\beta(x_2),\alpha(y_2)]]\\
&-&(-1)^{|x_2||y_2|}[\beta^2(y_2),[\beta(x_2),\alpha(y_3)]]\\
&=&[\beta^2(x_2),[\beta(y_2),\alpha(y_3)]]+(-1)^{|y_3|(|x_2|+|y_2|)}[\beta^2(y_3),[\beta(x_2),\alpha(y_2)]]\\
&+&(-1)^{|x_2|(|y_2|+|y_3|)}[\beta^2(y_2),[\beta(y_3),\alpha(x_2)]]\\
&=&0,
\end{eqnarray*}
since the identity of BiHom-super-Jacobi and equation \eqref{BiHom-skewsupersymmetry}. So $\tau(\beta(y_1))\tau(\beta^2(x_1))\Omega_1=0$.\\
For the same reason we show that each $\Omega_i$, where $1\leq i\leq6$, multiplied by their coefficient is zero. From where $L-R=0$. The theorem is proved.
\end{proof}

\section{Derivations and $(\alpha^{s},\beta^{r}) $-derivations of $3$-BiHom-Lie superalgebras }\label{Deriv}
In this section, we introduce the notion of $(\alpha^{s},\beta^{r}) $-derivations of $3$-BiHom-Lie superalgebras generalize the notion of  $(\alpha^{s},\beta^{r}) $-derivations of $3$-BiHom-Lie algebras introduced in \cite{Ab&Elh&Kay&Mak} and we give some results.
\begin{defn}
	Let $ (\mathfrak{g}, [\cdot,\cdot,\cdot], \alpha, \beta ) $ be a $ 3 $-BiHom-Lie superalgebra. A  linear map  $\mathfrak{D}: \mathfrak{g} \rightarrow \mathfrak{g}$ is  a derivation if it satisfait  for all  $x,y,z \in \mathcal{H}(\mathfrak{g})$ :
	\begin{eqnarray}
&&\mathfrak{D}\circ\alpha=\alpha\circ\mathfrak{D},\ \ \mathfrak{D}\circ\beta=\beta\circ\mathfrak{D},\nonumber\\
	\mathfrak{D}([x,y,z])&=&[\mathfrak{D}(x),y,z]+ (-1)^{|x||\mathfrak{D}|}[x,\mathfrak{D}(y),z]+(-1)^{|\mathfrak{D}|(|x|+|y|)}[x,y,\mathfrak{D}(z)] ,
	\end{eqnarray}
	and it is called   an $(  \alpha^{s},\beta^{r}) $-derivation of $ (\mathfrak{g}, [\cdot,\cdot,\cdot], \alpha, \beta ) $, if it satisfies :
\begin{eqnarray}	&&\mathfrak{D}\circ\alpha=\alpha\circ\mathfrak{D},\ \ \mathfrak{D}\circ\beta=\beta\circ\mathfrak{D},\nonumber\\	 \mathfrak{D}([x,y,z])&=&[\mathfrak{D}(x),\alpha^{s}\beta^{r}(y),\alpha^{s}\beta^{r}(z)]+(-1)^{|x||\mathfrak{D}|} [\alpha^{s}\beta^{r}(x),\mathfrak{D}(y),\alpha^{s}\beta^{r}(z)] \nonumber\\
	&&+(-1)^{|\mathfrak{D}|(|x|+|y|)}[\alpha^{s}\beta^{r}(x),\alpha^{s}\beta^{r}(y),\mathfrak{D}(z)]\label{DRS} .
	\end{eqnarray}
\end{defn}

%\begin{definition}
%	Let $ (\mathfrak{g}, [\cdot,\cdot,\cdot], \alpha, \beta ) $ be a $ 3 $-Leibniz-BiHom-Lie algebra and let $ D $ be a linear map of $ \mathfrak{g} $. For any positive integers $ r,s $, $ D $ is called a $ \beta^{r} \alpha^{s} $-derivation of $ (\mathfrak{g}, [\cdot,\cdot,\cdot], \alpha, \beta ) $, if :
%	\begin{itemize}
%		\item $ D \circ \alpha = \alpha \circ D $, and  $ D \circ \beta = \beta \circ D $.
%		\item $ D([x,y,z])=[D(x),\beta^{r}(y),\alpha^{s}(z)]+ [\beta^{r}(x),D(y),\alpha^{s}(z)]+[\beta^{r}(x),\alpha^{s}(y),D(z)] $, for all $ x,y,z \in \mathfrak{g} $.
%	\end{itemize}
%\end{definition}

Let   ${\rm Der}_{(  \alpha^{s},\beta^{r})}(\mathfrak{g})$ be the set of $(  \alpha^{s},\beta^{r}) $-derivations of $\mathfrak{g}$ and
set \[{\rm Der}(\mathfrak{g}):=\bigoplus_{s ,r\geq 0} {\rm Der}_{(  \alpha^{s},\beta^{r})}(\mathfrak{g}).\]

We show that ${\rm Der}(\mathfrak{g})$  is equipped with a  Lie superalgebra structure. In fact, for $ \mathfrak{D} \in {\rm Der}_{(  \alpha^{s},\beta^{r})}(\mathfrak{g}) $ and $ \mathfrak{D}' \in {\rm Der}_{(  \alpha^{s'},\beta^{r'})}(\mathfrak{g}) $, we have $ [\mathfrak{D},\mathfrak{D}'] \in {\rm Der}_{(  \alpha^{s+s'},\beta^{r+r'})}(\mathfrak{g}) $ , where $ [\mathfrak{D},\mathfrak{D}'] $ is the standard supercommutator defined by
$[\mathfrak{D},\mathfrak{D}'] = \mathfrak{D}\mathfrak{D}'-(-1)^{|\mathfrak{D}||\mathfrak{D}'|}\mathfrak{D}'\mathfrak{D} . $
\begin{defn}
	An endomorphism $ \mathfrak{D}$ of   a $ 3$-BiHom-Lie superalgebra $\mathfrak{g}$  is called  $(  \alpha^{s},\beta^{r}) $-quasiderivation if there exists an endomorphism  $ \mathfrak{D}' $ of $\mathfrak{g}$ such that
	\begin{align*}
	& \mathfrak{D} \circ \alpha = \alpha \circ \mathfrak{D};\  \mathfrak{D} \circ \beta = \beta \circ \mathfrak{D} \\
	& \mathfrak{D}' \circ \alpha = \alpha \circ \mathfrak{D}';\  \mathfrak{D}' \circ \beta = \beta \circ \mathfrak{D}'.
	\end{align*}
\begin{eqnarray}
	\mathfrak{D'}([x,y,z])&=&[\mathfrak{D}(x),\alpha^{s}\beta^{r}(y),\alpha^{s}\beta^{r}(z)]+ (-1)^{|x||\mathfrak{D}|}[\alpha^{s}\beta^{r}(x),\mathfrak{D}(y),\alpha^{s}\beta^{r}(z)]\\
&&+(-1)^{|\mathfrak{D}|(|x|+|y|)}[\alpha^{s}\beta^{r}(x),\nonumber
\alpha^{s}\beta^{r}(y),\mathfrak{D}(z)],
	\end{eqnarray}
	for any $ x,y,z\in \mathfrak{g}. $
\end{defn}
Then we define
\[{\rm QDer}(\mathfrak{g}):=\bigoplus_{s,r \geq 0} {\rm QDer}_{(  \alpha^{s},\beta^{r})}(\mathfrak{g}).\]

%\begin{proof}
%Straightforward.
%\end{proof}
\begin{prop}\label{AAA}
	Let $ (\mathfrak{g}, [\cdot,\cdot], \alpha, \beta ) $ be a nonmultiplicative BiHom-Lie superalgebra.   Let $\mathfrak{D}: \mathfrak{g} \rightarrow \mathfrak{g}$ be an $(\alpha^s,\beta^r)$-derivation of $ (\mathfrak{g}, [\cdot,\cdot], \alpha, \beta ) $.  If the following identity holds, for all $ x,y,z \in \mathcal{H}(\mathfrak{g}),$
	$$\circlearrowleft_{x,y,z}(-1)^{|x||z|} \tau(\mathfrak{D}(x))[y, z]
=0\;\;\;and\;\;\;\tau \circ \alpha^s\beta^r = \tau,$$
	then  $\mathfrak{D}$ is an $(\alpha^s,\beta^r)$-derivation of the induced nonmultiplicative $3$-BiHom-Lie superalgebra $ \mathfrak{g}_\tau $.
\end{prop}

\begin{proof}
	
	We have  to prove that
\begin{eqnarray*}
	\mathfrak{D}[x,y,z]_\tau&=&[\mathfrak{D}(x), \alpha^s\beta^r(y),  \alpha^s\beta^r(z)]_\tau
 +(-1)^{|x||\mathfrak{D}|}[ \alpha^s\beta^r(x), \mathfrak{D}(y),  \alpha^s\beta^r(z) ]_\tau\\&+&
(-1)^{|\mathfrak{D}|(|x|+|y|)} [ \alpha^s\beta^r(x),  \alpha^s\beta^r(y), \mathfrak{D}(z)]_\tau.
\end{eqnarray*}
	By applying $\mathfrak{D}$ to each side of equation (\ref{crochet_n}), we get

	\begin{eqnarray*}
		\mathfrak{L}&=&\mathfrak{D}[x,y,z]_\tau\\
&=& \tau(x)\mathfrak{D}[y,z] -(-1)^{|x||y|}\tau(y)\mathfrak{D}[x,z] +(-1)^{|z|(|x|+|y|)}\tau(z)\mathfrak{D}[x,y] \\
		&=&\tau(x)([\mathfrak{D}(y), \alpha^s\beta^r(z)] +(-1)^{|y||\mathfrak{D}|} [\alpha^s\beta^r(y), \mathfrak{D}(z)])\\
		&-& (-1)^{|x||y|} \tau(y)([\mathfrak{D}(x), \alpha^s\beta^r(z)] -(-1)^{|x|(|y|+|\mathfrak{D}|)} [\alpha^s\beta^r(x), \mathfrak{D}(z)])  \\
		&+&(-1)^{|z|(|x|+|y|)}        \tau(z)([\mathfrak{D}(x), \alpha^s\beta^r(y)]
+(-1)^{|z|(|x|+|y|)+|x||\mathfrak{D}|} [\alpha^s\beta^r(x), \mathfrak{D}(y)]),
	\end{eqnarray*}
	while,
	\begin{eqnarray*}
		 \mathfrak{R} &=& [\mathfrak{D}(x), \alpha^s\beta^r(y),  \alpha^s\beta^r(z)]_\tau
+(-1)^{|x||\mathfrak{D}|}[ \alpha^s\beta^r(x), \mathfrak{D}(y),  \alpha^s\beta^r(z) ]_\tau\\&+&
(-1)^{|\mathfrak{D}|(|x|+|y|)} [ \alpha^s\beta^r(x),  \alpha^s\beta^r(y), \mathfrak{D}(z)]_\tau\\
		&=&\tau(\mathfrak{D}(x))[\alpha^s\beta^r(y), \alpha^s\beta^r(z)] -(-1)^{|y|(|x|+|\mathfrak{D}|)}  \tau(\alpha^s\beta^r(y))[\alpha^s\beta^r(x), \mathfrak{D}(z)]\\&+&(-1)^{|z|(|x|+|y|+|\mathfrak{D}|)}  \tau(\alpha^s\beta^r(z))[\mathfrak{D}(x),\alpha^s\beta^r(y)]
		+(-1)^{|x||\mathfrak{D}|}\tau(\alpha^s\beta^r(x))[\mathfrak{D}(y), \alpha^s\beta^r(z)] \\
&-&(-1)^{|x||y|}  \tau(\mathfrak{D}(y))[\alpha^s\beta^r(x), \alpha^s\beta^r(z)]+ (-1)^{|z|(|x|+|y|)+|\mathfrak{D}|(|x|+|z|)}\tau(\alpha^s\beta^r(z))[\alpha^s\beta^r(x), \mathfrak{D}(y)]
		\\&+&(-1)^{|\mathfrak{D}|(|x|+|y|)}\tau(\alpha^s\beta^r(x))[ \alpha^s\beta^r(y), \mathfrak{D}(z)]
 -(-1)^{|\mathfrak{D}|(|x|+|y|)+|x||y|}  \tau(\alpha^s\beta^r(y))[\alpha^s\beta^r(x), \mathfrak{D}(z)] \\
&+&(-1)^{|z|(|x|+|y|)}  \tau(\mathfrak{D}(z))[\alpha^s\beta^r(x), \alpha^s\beta^r(y)].
	\end{eqnarray*}
	Using  the fact that \[ \tau \circ \alpha^s\beta^r = \tau , \]
	we can rewrite the right hand side as
	\begin{eqnarray*}
		\mathfrak{R} &=&\tau(\mathfrak{D}(x))[\alpha^s\beta^r(y), \alpha^s\beta^r(z)] -(-1)^{|y|(|x|+|\mathfrak{D}|)}  \tau(y)[\alpha^s\beta^r(x), \mathfrak{D}(z)]\\&+&(-1)^{|z|(|x|+|y|+|\mathfrak{D}|)}  \tau(z)[\mathfrak{D}(x),\alpha^s\beta^r(y)]
		+(-1)^{|x||\mathfrak{D}|}\tau(x)[\mathfrak{D}(y), \alpha^s\beta^r(z)] \\
&-&(-1)^{|x||y|}  \tau(\mathfrak{D}(y))[\alpha^s\beta^r(x), \alpha^s\beta^r(z)]+ (-1)^{|z|(|x|+|y|)+
|\mathfrak{D}|(|x|+|z|)}\tau(z)[\alpha^s\beta^r(x), \mathfrak{D}(y)]
		\\&+&(-1)^{|\mathfrak{D}|(|x|+|y|)}\tau(x)[ \alpha^s\beta^r(y), \mathfrak{D}(z)]
 -(-1)^{|\mathfrak{D}|(|x|+|y|)+|x||y|}  \tau(y)[\alpha^s\beta^r(x), \mathfrak{D}(z)] \\
&+&(-1)^{|z|(|x|+|y|)}  \tau(\mathfrak{D}(z))[\alpha^s\beta^r(x), \alpha^s\beta^r(y)].
	\end{eqnarray*}
We use the property $\tau|_{\mathfrak{g}_1}\equiv0$ we find the following equality\\
	\begin{eqnarray*}
		\mathfrak{R}-\mathfrak{L} &=&\tau(\mathfrak{D}(x))[\alpha^s\beta^r(y), \alpha^s\beta^r(z)]
-(-1)^{|x||y|}  \tau(\mathfrak{D}(y))[\alpha^s\beta^r(x), \alpha^s\beta^r(z)] \\
&+&(-1)^{|z|(|x|+|y|)}  \tau(\mathfrak{D}(z))[\alpha^s\beta^r(x), \alpha^s\beta^r(y)]\\
&=&\circlearrowleft_{x,y,z}(-1)^{|x||z|} \tau(\mathfrak{D}(x))[\alpha^s\beta^r(y), \alpha^s\beta^r(z)].
	\end{eqnarray*}
	We then obtain the result by assuming that the following identity holds, $\forall x,y,z \in \mathcal{H}(\mathfrak{g})$,
{\small$$\circlearrowleft_{x,y,z}(-1)^{|x||z|} \tau(\mathfrak{D}(x))[\alpha^s\beta^r(y), \alpha^s\beta^r(z)]=\circlearrowleft_{x,y,z}(-1)^{|x||z|} \tau(\mathfrak{D}(\alpha^s\beta^r(x)))[\alpha^s\beta^r(y), \alpha^s\beta^r(z)]
=0.$$	}
	This ends the proof.
\end{proof}
\begin{prop}\label{AAA}
	Let $ (\mathfrak{g}, [\cdot,\cdot], \alpha, \beta ) $ be a nonmultiplicative
BiHom-Lie superalgebra.   Let $\mathfrak{D}$ be an $(\alpha^s,\beta^r)$-quasiderivation of $ (\mathfrak{g}, [\cdot,\cdot], \alpha, \beta ) $.  If the following identity holds, for all $ x,y,z \in \mathcal{H}(\mathfrak{g}),$
	$$\circlearrowleft_{x,y,z}(-1)^{|x||z|} \tau(\mathfrak{D}(x))[y, z]
=0\;\;\;and\;\;\;\tau \circ \alpha^s\beta^r = \tau$$
	then  $\mathfrak{D}$ is an $(\alpha^s,\beta^r)$-super-quasiderivation of the induced ternary BiHom-Lie superalgebra $ \mathfrak{g}_\tau $.
\end{prop}

\section{Rota-Baxter of  $3$-BiHom-Lie superalgebras}
In this section, we give the definition of Rota-Baxter of  $3$-BiHom-Lie superalgebras and the realizations of Rota-Baxter of $3$-BiHom-Lie superalgebras from Rota-Baxter of BiHom-Lie superalgebras.
\begin{defn}\label{definition and notation}
Let $(\mathfrak{g}, [\cdot,\cdot],\alpha,\beta)$ be a BiHom-Lie superalgebra  and $\lambda\in \mathbb{K}$.
If an even linear map $R: \mathfrak{g}\rightarrow \mathfrak{g}$ satisfies
\begin{eqnarray}
&&\ R\circ\alpha=\alpha\circ R\ \textrm{and}\ R\circ\beta=\beta\circ R,\\
&&\ [R(x),R(y)]=R([R(x),y]+[x,R(y)]+\lambda[x,y]),
\end{eqnarray}
 for all $ x, y \in \mathfrak{g}$, then $R$ is called a  Rota-Baxter operator of weight $\lambda$ on $(\mathfrak{g}, [.,.],\alpha,\beta)$ .
\end{defn}
\begin{defn}
Let $(\mathfrak{g}, [\cdot,\cdot,\cdot],\alpha,\beta)$ be a $3$-BiHom-Lie superalgebra  and $\lambda\in \mathbb{K}$.
If an even linear map $R: \mathfrak{g}\rightarrow \mathfrak{g}$ satisfies
\begin{align}&\ R\circ\alpha=\alpha\circ R\ \textrm{and}\ R\circ\beta=\beta\circ R,\\
\ [R(x),R(y),R(z)]&=R([R(x),R(y),z]+[R(x),y,R(z)]+[x,R(y),R(z)]\nonumber\\&+\lambda [R(x),y,z]+\lambda[x,R(y),z]+\lambda[x,y,R(z)]+\lambda^2[x,y,z]),
\end{align}
 for all $ x, y,z \in \mathfrak{g}$, then $R$ is called a  Rota-Baxter operator of weight $\lambda$ on $(\mathfrak{g}, [\cdot,\cdot,\cdot],\alpha,\beta)$ .
\end{defn}
\begin{prop}
  Let $(\mathfrak{g}, [\cdot ,\cdot ,\cdot ],\alpha,\beta)$  be a $3$-BiHom-Lie superalgebra over $\mathbb{K}$. An invertible linear mapping $R : \mathfrak{g} \rightarrow \mathfrak{g}$
is a Rota-Baxter operator of weight $0$ on $\mathfrak{g}$ if and only if $R^{-1}$ is an even derivation
on $\mathfrak{g}$.
\end{prop}
\begin{proof}

$R$ is an even invertible Rota-Baxter operator of weight $0$ on $\mathfrak{g}$, if and only if:\\
$[R(x),R(y),R(z)]=R([R(x),R(y),z]+[R(x),R(y),z]+[x,R(y),R(z)]),\;\forall x,y,z\in \mathfrak{g}$.\\

Suppose that $X=R(x),\;Y=R(y)$ and $Z=R(z)$, we have:\\

$[X,Y,Z]=R([X,Y,R^{-1}(Z)]+[X,R^{-1}(Y),Z]+[R^{-1}(X),Y,Z])$, which gives\\
$R^{-1}([X,Y,Z])=[X,Y,R^{-1}(Z)]+[X,R^{-1}(Y),Z]+[R^{-1}(X),Y,Z]$. Thus $R^{-1}$ is an even derivation on $\mathfrak{g}$.
\end{proof}
\begin{prop}
   Let $R$ be a Rota-Baxter of weight $\lambda$ on a nonmultiplicative BiHom-Lie superalgebra $(\mathfrak{g},[\cdot,\cdot],\alpha,\beta)$ and $\tau\in\mathfrak g^*$ satisfying to the two conditions \eqref{ConditionTau1} and \eqref{ConditionTau2}. Then $R$ is a Rota-Baxter operator of weight $\lambda$ on the induced nonmultiplicative $3$-BiHom-Lie superalgebra  $\mathfrak{g}_\tau$
 if and only if $R$ satisfies
 \begin{equation}\label{RotaBaxterCondition}
  \displaystyle\circlearrowleft_{x,y,z}(-1)^{|x||z|}\tau(x)[R(y),R(z)]\in \ker(R-\lambda Id_{\mathfrak{g}}),\; \forall \ x,y,z\in\mathcal{H}(\mathfrak{g}).
 \end{equation}
\end{prop}
\begin{proof}

Let $x,y,z\in\mathcal{H}(\mathfrak{g})$, we have
{\small\begin{align*}
[R(x),R(y),R(z)]_\tau =&\tau(R(x))[R(y),R(z)]-(-1)^{|x||y|}\tau(R(y))[R(x),R(z)]\\&+(-1)^{|z|(|x|+|y|)}\tau(R(z))[R(x),R(y)]\\
=&\tau(R(x))R\Big([R(y),z]+[y,R(z)]-\lambda[y,z]\Big)\\&-(-1)^{|x||y|}\tau(R(y))R\Big([R(x),z]+[x,R(z)]-\lambda[x,z]\Big)\\
&+(-1)^{|z|(|x|+|y|)}\tau(R(z))R\Big([R(x),y]+[x,R(y)]-\lambda[x,y]\Big)\\
=&R\Big(\tau(R(x))[R(y),z]-(-1)^{|x||y|}\tau(R(y))[R(x),z]\Big)\\&+R\Big(\tau(R(x))[y,R(z)]+(-1)^{|z|(|x|+|y|)}\tau(R(z))[R(x),y]\Big)\\
&+R\Big(-(-1)^{|x||y|}\tau(R(y))[x,R(z)]+(-1)^{|z|(|x|+|y|)}\tau(R(z))[x,R(y)]\Big)\\
&-\lambda R\Big(\tau(R(x))[y,z]\Big)+(-1)^{|x||y|}\lambda R\Big(\tau(R(y))[x,z]\Big)\\&-(-1)^{|z|(|x|+|y|)}\lambda R\Big(\tau(R(z))[x,y]\Big)\\
=&R\Big([R(x),R(y),z]_\tau+[R(x),y,R(z)]_\tau+[x,R(y),R(z)]_\tau\Big)\\&-\lambda R\Big([R(x),y,z]_\tau+[x,R(y),z]_\tau+[x,y,R(z)]_\tau\Big)\\
&+\lambda^2R([x,y,z]_\tau)+A
\end{align*}}
where
{\small\begin{align*}
A=&\lambda\tau(x)\Big(R\big([R(y),z]+[y,R(z)]\big)\Big)-(-1)^{|x||y|}\lambda\tau(y)\Big(R\big([R(x),z]+[x,R(z)]\big)\Big)\\
&+(-1)^{|z|(|x|+|y|)}\lambda\tau(z)\Big(R\big([R(x),y]+[x,R(y)]\big)\Big)\\&-\lambda^2\Big(\tau(x)[y,z]+(-1)^{|x||y|}\tau(y)[x,z]
-(-1)^{|z|(|x|+|y|)}\tau(z)[x,y]\Big)\\
&-R\big(\tau(x)[R(y),R(z)]-(-1)^{|x||y|}\tau(y)[R(x),R(z)]+(-1)^{|z|(|x|+|y|)}\tau(z)[R(x),R(y)]\big)\\
=&\lambda\tau(x)\Big([R(y),R(z)]+\lambda R([y,z])\Big)-(-1)^{|x||y|}\lambda\tau(y)\Big([R(x),R(z)]+\lambda R([x,z])\Big)\\
&+(-1)^{|z|(|x|+|y|)}\lambda\tau(z)\Big([R(x),R(y)]+\lambda R([x,R(y)])\Big)\\&-\lambda^2\Big(\tau(x)[y,z]+(-1)^{|x||y|}\tau(y)[x,z]
-(-1)^{|z|(|x|+|y|)}\tau(z)[x,y]\Big)\\
&-R\big(\tau(x)[R(y),R(z)]-(-1)^{|x||y|}\tau(y)[R(x),R(z)]+(-1)^{|z|(|x|+|y|)}\tau(z)[R(x),R(y)]\big)\\
=&\lambda\tau(x)[R(y),R(z)]-(-1)^{|x||y|}\lambda\tau(y)[R(x),R(z)]+(-1)^{|z|(|x|+|y|)}\lambda\tau(z)[R(x),R(y)]\\
&-R\big(\tau(x)[R(y),R(z)]-(-1)^{|x||y|}\tau(y)[R(x),R(z)]+(-1)^{|z|(|x|+|y|)}\tau(z)[R(x),R(y)]\big)\\
=&(\lambda Id_{\mathfrak{g}}-R)\Big(\tau(x)[R(y),R(z)]-(-1)^{|x||y|}\tau(y)[R(x),R(z)]+(-1)^{|z|(|x|+|y|)}\tau(z)[R(x),R(y)]\Big)\\
=&(\lambda Id_{\mathfrak{g}}-R) \displaystyle\circlearrowleft_{x,y,z}(-1)^{|x||z|}\tau(x)[R(y),R(z)],
\end{align*}}
thus, $R$ is a Rota-Baxter operator on $(\mathfrak{g},[.,.,.],\alpha,\beta)$ if and only if $A=0$.
\end{proof}

Let $R$ be a Rota-Baxter operator of weight $\lambda$ on a $3$-BiHom-Lie superalgebra $(\mathfrak{g},[\cdot,\cdot,\cdot],\alpha,\beta)$, we define a ternary
operation on $\mathfrak{g}$ by
\begin{eqnarray*}
[x_1,x_2,x_3]_R&=&[R(x_1),R(x_2),x_3]+[R(x_1),x_2,R(x_3)]+[x_1,R(x_2),R(x_3)]+\lambda [R(x_1),x_2,x_3]\\&+&\lambda[x_1,R(x_2),x_3]+
\lambda[x_1,x_2,R(x_3)]+\lambda^2[x_1,x_2,x_3]\\
&=&\displaystyle\sum_{\emptyset\neq I\subseteq[3]}\lambda^{|I|-1}[\widehat{R}_I(x_1), \widehat{R}_I(x_2), \widehat{R}_I(x_3)],
\end{eqnarray*}for all $x_1,x_2,x_3\in\mathcal{H}(\mathfrak{g})$,
where $\widehat{R}_I(x_i)=\left\{
                       \begin{array}{ll}
                         x_i &\; i\;\in I \hbox{ ;} \\
                         R(x_i) & \; i\;\not\in I\hbox{.}
                       \end{array}
                     \right.$\\
Then we have the following result.

\begin{thm}
Let $R$ be a Rota-Baxter operator of weight $\lambda$ on a $3$-BiHom-Lie superalgebra $(\mathfrak{g},[\cdot,\cdot,\cdot],\alpha,\beta)$. Then $(\mathfrak{g},[\cdot,\cdot,\cdot]_R,\alpha,\beta)$ is a $3$-BiHom-Lie superalgebra and $R$ be a Rota-Baxter operator of weight $\lambda$ on  $(\mathfrak{g},[\cdot,\cdot,\cdot]_R,\alpha,\beta)$.
\end{thm}
\begin{proof}
It is clear that $[\cdot,\cdot ,\cdot ]_R$ is  super-skewsymetric.

Let $x_1, x_2, x_3, x_4, x_5\in\mathcal{H}(\mathfrak{g})$. Denote $y_1=\beta^2(x_1),\;y_2=\beta^2(x_2)$ and $y_3 = [\beta(x_3), \beta(x_4),\alpha( x_5)]_R$,
we have

$[\beta^2(x_1),\beta^2(x_2),[\beta(x_3),\beta(x_4),\alpha(x_5)]_R]_R$
\begin{eqnarray*}
&=&[y_1,y_2,y_3]_R
\\&=&\displaystyle\sum_{\emptyset\neq I\subseteq[3]}\lambda^{|I|-1}[\widehat{R}_I(y_1),\widehat{R}_I(y_2),\widehat{R}_I(y_3)]\\
&=&\displaystyle\sum_{\emptyset\neq I\subseteq[3],\;3\not\in I}\lambda^{|I|-1}[\widehat{R}_I(y_1),\widehat{R}_I(y_2),R_I(y_3)]
+\displaystyle\sum_{\emptyset\neq I\subseteq[3],\;3\in I}\lambda^{|I|-1}[\widehat{R}_I(y_1),\widehat{R}_I(y_2),y_3]\\
&=&\displaystyle\sum_{\emptyset\neq I\subseteq[3],\;3\not\in I}\lambda^{|I|-1}[\widehat{R}_I(y_1),\widehat{R}_I(y_2),[R(\beta(x_3),R(\beta(x_4)),R(\alpha(x_5))]]\\
&+&\displaystyle\sum_{\emptyset\neq I\subseteq[3],\;3\in I}\lambda^{|I|-1}
\Big[\widehat{R}_I(y_1),\widehat{R}_I(y_2),\displaystyle\sum_{\emptyset\neq J\subseteq \mathcal{P}(\{3,4,5\})}
\lambda^{|J|-1}[\widehat{R}_J(\beta(x_3)),\widehat{R}_J(\beta(x_4)),\widehat{R}_J(\beta(x_5))]\Big]\\
&=&\displaystyle\sum_{\emptyset\neq K\subseteq[5],\;K\cap\mathcal{P}(\{3,4,5\})=\emptyset}\lambda^{|K|-1}
[\widehat{R}_K(\beta^2(x_1)),\widehat{R}_K(\beta^2(x_2)),[\widehat{R}_K(\beta(x_3)),\widehat{R}_K(\beta(x_4)),\widehat{R}_K(\alpha(x_5))]]\\
&+&\displaystyle\sum_{\emptyset\neq I\subseteq[5],\;I\cap\mathcal{P}(\{3,4,5\})\neq\emptyset}\lambda^{|I|-1}
[\widehat{R}_I(\beta^2(x_1)),\widehat{R}_I(\beta^2(x_2)),[\widehat{R}_I(\beta(x_3)),\widehat{R}_I(\beta(x_4)),\widehat{R}_I(\alpha(x_5))]]\\
&=&\displaystyle\sum_{\emptyset\neq I\subseteq[5]}\lambda^{|I|-1}
[\widehat{R}_I(\beta^2(x_1)),\widehat{R}_I(\beta^2(x_2)),[\widehat{R}_I(\beta(x_3)),\widehat{R}_I(\beta(x_4)),\widehat{R}_I(\alpha(x_5))]].
\end{eqnarray*}
Since $(\mathfrak{g},[\cdot,\cdot,\cdot],\alpha,\beta)$ is a $3$-BiHom-Lie superalgebra, $R\circ\alpha=\alpha\circ R$ and $R\circ\beta=\beta\circ R$, for any given
$\emptyset\neq I\subseteq[5]$, we have\\
$[\widehat{R}_I(\beta^2(x_1)),\widehat{R}_I(\beta^2(x_2)),[\widehat{R}_I(\beta(x_3),\widehat{R}_I(\beta(x_4)),\widehat{R}_I(\alpha(x_5))]]$
\begin{eqnarray*}
&=&(-1)^{(|x_4|+|x_5|)(|x_1|+|x_2|+|x_3|)}
[\widehat{R}_I(\beta^2(x_4)),\widehat{R}_I(\beta^2(x_5)),[\widehat{R}_I(\beta(x_1)),\widehat{R}_I(\beta(x_2)),\widehat{R}_I(\alpha(x_3))]]\\
&-&(-1)^{(|x_3|+|x_5|)(|x_1|+|x_2|)+|x_4||x_5|}
[\widehat{R}_I(\beta^2(x_3)),\widehat{R}_I(\beta^2(x_5)),[\widehat{R}_I(\beta(x_1)),\widehat{R}_I(\beta(x_2)),\widehat{R}_I(\alpha(x_4))]]\\
&+&(-1)^{(|x_1|+|x_2|)(|x_3|+|x_4|)}
[\widehat{R}_I(\beta^2(x_3)),\widehat{R}_I(\beta^2(x_4)),[\widehat{R}_I(\beta(x_1)),\widehat{R}_I(\beta(x_2)),\widehat{R}_I(\alpha(x_5))]].
\end{eqnarray*}
Thus from the above sum, we conclude that $(\mathfrak{g},[~,~,~]_R,\alpha,\beta)$ is a BiHom 3-Lie superalgebra.\\

It remains to show that $R$ is a Rota-Baxter operator of weight $\lambda$ on  $(\mathfrak{g},[~,~,~]_R,\alpha,\beta)$.\\
\begin{eqnarray*}
[R(x_1),R(x_2),R(x_3)]_R\Big)&=&\displaystyle\sum_{\emptyset\neq I\subseteq[3]}\lambda^{|I|-1}[\widehat{R}(R(x_1)),\widehat{R}(R(x_2)),\widehat{R}(R(x_3))]\\
&=&\displaystyle\sum_{\emptyset\neq I\subseteq[3]}\lambda^{|I|-1}[R(\widehat{R}(x_1)),R(\widehat{R}(x_2)),R(\widehat{R}(x_3))]\\
&=&\displaystyle\sum_{\emptyset\neq I\subseteq[3]}\lambda^{|I|-1}R([\widehat{R}(x_1),\widehat{R}(x_2),\widehat{R}(x_3)])\\
&=&R\Big(\displaystyle\sum_{\emptyset\neq I\subseteq[3]}\lambda^{|I|-1}[\widehat{R}(x_1),\widehat{R}(x_2),\widehat{R}(x_3)]\Big),
\end{eqnarray*}
which gives the requested result. The theorem is proved.
\end{proof}
\begin{prop}
\label{BiHomrotabaxter}
Let $(\mathfrak{g},[\cdot,\cdot,\cdot],\alpha,\beta)$ be a $3$-BiHom-Lie superalgebra and $R$ be a Rota-Baxter operator of weight $\lambda$ on $\mathfrak{g}$ such that
$R^2=R$, then $(\mathfrak{g},[\cdot,\cdot,\cdot]_R,\alpha\circ R,\beta\circ  R)$ is a noncommutative $3$-BiHom-Lie superalgebra.

\end{prop}

\section{Nijenhuis Operators on BiHom $3$-Lie superalgebras}

In this section, we study the second order deformation of $3$-BiHom-Lie superalgebras, and introduce the notion of Nijenhuis operator on $3$-BiHom-Lie superalgebras, which could generate a trivial deformation. In the other part of this section we give some properties and results of Nijenhuis operators.

\subsection{Second-order deformation of $3$-BiHom-Lie superalgebras}
Let $(\mathfrak{g},[\cdot,\cdot,\cdot],\alpha,\beta)$ be a $3$-BiHom-Lie superalgebras and $\omega_i:\mathfrak{g}\times\mathfrak{g}\times\mathfrak{g}\longrightarrow\mathfrak{g},\;i=1,2$ be a super-skewsymetric multilinear maps. Consider a $\lambda$-parametrized
family of $3$-linear operations:
\begin{equation}\label{omegaparam}
[x_1,x_2,x_3]_\lambda=[x_1,x_2,x_3]+\lambda\omega_1(x_1,x_2,x_3)+\lambda^2\omega_2(x_1,x_2,x_3),
\end{equation}
where $\lambda\in \mathbb{K}$. If all $[\cdot,\cdot,\cdot]_\lambda$ are $3$-BiHom-Lie superalgebra structures, we say
that $\omega_1,\;\omega_2$ generate an second-order $1$-parameter deformation of the $3$-BiHom-Lie superalgebra $(\mathfrak{g},[\cdot,\cdot,\cdot],\alpha,\beta)$.
\begin{prop}
With the above notations, $\omega_1,\;\omega_2$ generate a second-order $1$-parameter deformation of the $3$-BiHom-Lie superalgebra $(\mathfrak{g},\omega_0=[\cdot,\cdot,\cdot],\alpha,\beta)$ if and only if for all $i,j=1,2$ and $l=1,\cdots,4$ the following conditions are satisfied :
\begin{eqnarray}
&&\omega_i\circ\alpha^{\otimes3}=\alpha\circ\omega_i\;\; and\;\; \omega_i\circ\beta^{\otimes3}=\beta\circ\omega_i,\\
&&\displaystyle\sum_{i+j=l}\omega_i\circ_{\alpha,\beta}\omega_j=0,
\end{eqnarray}
where  $\omega_i\circ_{\alpha,\beta}\omega_j:\wedge^2\mathfrak{g}\otimes\wedge^2\mathfrak{g}\wedge\mathfrak{g}\longrightarrow\mathfrak{g}$ is defined by
\begin{eqnarray*}
\omega_i\circ_{\alpha,\beta}\omega_j(X,Y,z)&=&\omega_i(\omega_j(X,.)\ast_{\alpha,\beta} Y,\beta^2(z))-\omega_i(\widetilde{\beta}^2(X),\omega_j(\widetilde{\beta}(Y),\alpha(z)))\\
&+&(-1)^{|X||Y|}\omega_i(\widetilde{\beta}^2(Y),\omega_j(\widetilde{\beta}(X),\alpha(z))),
\end{eqnarray*}
where $\omega_j(X,.)\ast_{\alpha,\beta} Y\in\wedge^2\mathfrak{g}$ is given by
$$\omega_j(X,.)\ast_{\alpha,\beta} Y=\omega_j(\widetilde{\beta}(X),\alpha(y_1))\wedge\beta^2(y_2)+(-1)^{|y_1||X|}\beta^2(y_1)\wedge\omega_j(\widetilde{\beta}(X),\alpha(y_2))$$
\end{prop}
\begin{proof}
$(\mathfrak{g},\omega_{\lambda},\alpha,\beta)$ are $3$-BiHom-Lie superalgebra structures if and only if
\begin{equation}\label{alphbetaNijen}
\omega_\lambda\circ\alpha^{\otimes3}=\alpha\circ\omega_\lambda\;\; and\;\; \omega_\lambda\circ\beta^{\otimes3}=\beta\circ\omega_\lambda
\end{equation}
\begin{equation}\label{omegalambdaNijen}
\omega_\lambda(\widetilde{\beta}^2(X),\omega_\lambda(\widetilde{\beta}(Y),\alpha(z)))=\omega_\lambda(\omega_\lambda(X,.)\ast Y,\beta^2(z))
+(-1)^{|X||Y|}\omega_\lambda(\widetilde{\beta}^2(Y),\omega_\lambda(\widetilde{\beta}(X),\alpha(z)))
\end{equation}
By \eqref{alphbetaNijen}, we have
$$\omega_i\circ\alpha^{\otimes3}=\alpha\circ\omega_i\;\; and\;\; \omega_i\circ\beta^{\otimes3}=\beta\circ\omega_i.$$
Expending the equations in \eqref{omegalambdaNijen} and collecting coefficients of $\lambda^l$, we see that \eqref{omegalambdaNijen} is equivalent to the system of equation
{\small$$\displaystyle\sum_{i+j=l}\omega_i(\widetilde{\beta}^2(X),\omega_j(\widetilde{\beta}(Y),\alpha(z)))=
\displaystyle\sum_{i+j=l}\omega_i(\omega_j(X,.)\ast_{\alpha,\beta} Y,\beta^2(z))
+\displaystyle\sum_{i+j=l}(-1)^{|X||Y|}\omega_i(\widetilde{\beta}^2(Y),\omega_j(\widetilde{\beta}(X),\alpha(z))).$$
}Thus,we have
$$\displaystyle\sum_{i+j=l}\omega_i(\omega_j(X,.)\ast_{\alpha,\beta} Y,\beta^2(z))-\omega_i(\widetilde{\beta}^2(X),\omega_j(\widetilde{\beta}(Y),\alpha(z)))
+(-1)^{|X||Y|}\omega_i(\widetilde{\beta}^2(Y),\omega_j(\widetilde{\beta}(X),\alpha(z))).$$
\end{proof}
\begin{cor}
If $\omega_1,\omega_2$ generate a second-order $1$-parameter deformation of the $3$-BiHom-Lie superalgebra $(\mathfrak{g},\omega_0=[\cdot,\cdot,\cdot],\alpha,\beta)$, then $\omega_1$ is  $2$-cocycle of the $3$-BiHom-Lie superalgebra $(\mathfrak{g},\omega_0=[\cdot,\cdot,\cdot],\alpha,\beta)$ with the coefficients in the adjoint representation
\end{cor}
\begin{proof}
For $l=1$, the condition $(24)$ gives the following equality
$$\omega_0\circ_{\alpha,\beta}\omega_1+\omega_1\circ_{\alpha,\beta}\omega_0=0$$
which is equivalent to that $\omega_1$ is a $2$-cocycle.
\end{proof}
\begin{cor}
If $\omega_1,\omega_2$ generate a second-ordre $1$-parameter deformation of the $3$-BiHom-Lie superalgebra $(\mathfrak{g},\omega_0,\alpha,\beta)$, then $(\mathfrak{g},\omega_2,\alpha,\beta)$ is a $3$-BiHom-Lie superalgebra.
\end{cor}
\begin{proof}
By $(23)$, let $i=2$, we deduce that $$\omega_2\circ\alpha^{\otimes3}=\alpha\circ\omega_2\;\;and\;\;\omega_2\circ\beta^{\otimes3}=\beta\circ\omega_2$$
and by $(24)$, let $l=4$, we deduce that
$$\omega_2\circ_{\alpha,\beta}\omega_2=0$$
which equivalent to that $(\mathfrak{g},\omega_2,\alpha,\beta)$ is a $3$-BiHom-Lie superalgebra.
\end{proof}
\begin{defn}
A deformation is said to be trivial if there exists an even linear map $N:\mathfrak{g}\longrightarrow\mathfrak{g}$ such that for all $\lambda,\;T_\lambda=id+\lambda N$ satisfies
\begin{equation}\label{alphabetatrivialdeform}
T_\lambda\circ\alpha=\alpha\circ T_\lambda\;\;and\;\;T_\lambda\circ\beta=\beta\circ T_\lambda
\end{equation}
\begin{equation}\label{trivialdeformation}
T_\lambda[x_1,x_2,x_3]_\lambda=[T_\lambda x_1,T_\lambda x_2,T_\lambda x_3],\;\forall x_1,x_2,x_3\in\mathcal{H}(\mathfrak{g})
\end{equation}
\end{defn}
Eq. \eqref{alphabetatrivialdeform} equals to $$N\circ\alpha=\alpha\circ N\;\;and\;\;N\circ\beta=\beta\circ N$$
The left hand side of Eq. \eqref{trivialdeformation} equals to
\begin{eqnarray*}
&&[x_1,x_2,x_3]+\lambda(\omega_1(x_1,x_2,x_3)+N[x_1,x_2,x_3])\\&&+
\lambda^2(\omega_2(x_1,x_2,x_3)+N\omega_1(x_1,x_2,x_3))+\lambda^3N\omega_2(x_1,x_2,x_3).
\end{eqnarray*}
The right hand side of Eq. \eqref{trivialdeformation} equals to
\begin{eqnarray*}
&&[x_1,x_2,x_3]+
\lambda([Nx_1,x_2,x_3]+[x_1,Nx_2,x_3]
+[x_1,x_2,Nx_3])\\&&
+\lambda^2([Nx_1,Nx_2,x_3]+[Nx_1,x_2,Nx_3]+[x_1,Nx_2,Nx_3])+
\lambda^3[Nx_1,Nx_2,Nx_3].
\end{eqnarray*}
Therefore, by Eq. \eqref{trivialdeformation}, we have
{\small\begin{align}
&\omega_1(x_1,x_2,x_3)+N[x_1,x_2,x_3]=[Nx_1,x_2,x_3]+
[x_1,Nx_2,x_3]+[x_1,x_2,Nx_3],\\
&\omega_2(x_1,x_2,x_3)+N\omega_1(x_1,x_2,x_3)
=[Nx_1,Nx_2,x_3]+[Nx_1,x_2,Nx_3]+[x_1,Nx_2,Nx_3],\\
&N\omega_2(x_1,x_2,x_3)=[Nx_1,Nx_2,Nx_3].
\end{align}}
Let $(\mathfrak{g},[\cdot,\cdot,\cdot],\alpha,\beta)$ be a $3$-BiHom-Lie superalgebra, and $N:\mathfrak{g}\longrightarrow\mathfrak{g}$ a linear map. Define a $3$-ary bracket $[\cdot,\cdot,\cdot]^1_N:\wedge^3\mathfrak{g}\longrightarrow\mathfrak{g}$ by
\begin{eqnarray}
[x_1,x_2,x_3]^1_N&=&[Nx_1,x_2,x_3]+[x_1,Nx_2,x_3]+
[x_1,x_2,Nx_3]-
N[x_1,x_2,x_3].
\end{eqnarray}
Then we define $3$-ary bracket $[\cdot,\cdot,\cdot]^2_N:\wedge^3\mathfrak{g}\longrightarrow\mathfrak{g}$, via induction by
\begin{equation}
[x_1,x_2,x_3]^2_N=[Nx_1,Nx_2,x_3]+[Nx_1,x_2,Nx_3]+
[x_1,Nx_2),Nx_3]-N[x_1,x_2,x_3]^1_N.
\end{equation}
\begin{defn}
Let $(\mathfrak{g},[\cdot,\cdot,\cdot],\alpha,\beta)$ be a $3$-BiHom-Lie superalgebra. An even linear map $N:\mathfrak{g}\rightarrow\mathfrak{g}$ is called a Nijenhuis operator
if
\begin{eqnarray}\label{equNijenhuis}
[N(x_1),N(x_2),N(x_3)]&=&N([x_1,x_2,x_3]^2_N)\\
&=&\displaystyle\sum_{\emptyset\neq I\subseteq[3]}(-1)^{|I|-1}N^{|I|}[\widetilde{N}(x_1),\widetilde{N}(x_2),\widetilde{N}(x_3)]\nonumber
\end{eqnarray}
$\forall x_1,x_2,x_3\in\mathfrak{g}$, where
$\widetilde{N}(x_i)=\left\{
                       \begin{array}{ll}
                         x_i &\; i\;\in I \hbox{ ;} \\
                         N(x_i) & \; i\;\not\in I\hbox{.}
                       \end{array}
                     \right.$
\end{defn}
\subsection{Some properties of Nijenhuis operators}
Let $(\mathfrak{g},[\cdot,\cdot],\alpha,\beta)$ be a BiHom-Lie superalgebra, a  Nijenhuis operator of $\mathfrak{g}$ is a linear map $N:\mathfrak{g}\to \mathfrak{g}$ compatible with $\alpha$ and $\beta$ $($ i.e $N\circ\alpha=\alpha\circ N$ and $N\circ\beta=\beta\circ N$ $)$ defined by
\begin{equation}[N(x),N(y)]=N([N(x),y]+[x,N(y)]-N([x,y])),\;\forall\;x,y\in\mathfrak{g}.\end{equation}

\begin{prop}
   Let $N$ be a Nijenhuis operator on a BiHom-Lie superalgebra $(\mathfrak{g},[\cdot,\cdot],\alpha,\beta)$ and $\tau\in\wedge\mathfrak g^*$ satisfying to the two conditions \eqref{ConditionTau1} and \eqref{ConditionTau2}. Then $N$ is a Nijenhuis operator on the induced $3$-BiHom-Lie superalgebra   $\mathfrak{g}_\tau$.
\end{prop}
\begin{proof}
Let $x_1,x_2,x_3\in\mathcal{H}(\mathfrak{g})$. In the first hand, we have
\begin{eqnarray*}
[N(x_1),N(x_2),N(x_3)]_\tau&=&\tau(N(x_1))[N(x_2),N(x_3)]-(-1)^{|x_1||x_2|}\tau(N(x_2))[N(x_1),N(x_3)]\\
&+&(-1)^{|x_3|(|x_1|+|x_2|)}\tau(N(x_3))[N(x_1),N(x_2)]\\
&=&\tau(N(x_1))N[Nx_2,x_3]+\tau(N(x_1))N[x_2,x_3]\\&-&\tau(N(x_1))N^2[x_2,x_3]
-(-1)^{|x_1||x_2|}\big(\tau(N(x_2))N[Nx_1,x_3]\\&+&\tau(N(x_2))N[x_1,Nx_3]-
\tau(N(x_2))N^2[x_1,x_3]\big)\\
&+&(-1)^{|x_3|(|x_1|+|x_2|)}\big(\tau(N(x_3))N[Nx_1,x_2]+\tau(N(x_3))[x_1,Nx_2]\\&-&
\tau(N(x_3))N^2([x_1,x_2])\big)
\end{eqnarray*}
Since $N$ is a Nijenhuis operator on a BiHom-Lie superalgebra $(\mathfrak{g},[\cdot,\cdot],\alpha,\beta)$.\\
In the other hand, we have
\begin{eqnarray*}
N([N(x_1),N(x_2),x_3]_\tau)&=&\tau(N(x_1))N([N(x_2),x_3])-(-1)^{|x_1||x_2|}\tau(N(x_2))N([N(x_1),x_3])\\
&+&(-1)^{|x_3|(|x_1|+|x_2|)}\tau(x_3)N([N(x_1),N(x_2)]).
\end{eqnarray*}
\begin{eqnarray*}
N([N(x_1),x_2,N(x_3)]_\tau)&=&\tau(N(x_1))N([x_2,N(x_3)])-(-1)^{|x_1||x_2|}\tau(x_2)N([N(x_1),N(x_3)])\\
&+&(-1)^{|x_3|(|x_1|+|x_2|)}\tau(N(x_3))N([N(x_1),x_2]).
\end{eqnarray*}
\begin{eqnarray*}
N([x_1,N(x_2),N(x_3)]_\tau)&=&\tau(x_1)N([N(x_2),N(x_3)])-(-1)^{|x_1||x_2|}\tau(N(x_2))N([x_1,N(x_3)])\\
&+&(-1)^{|x_3|(|x_1|+|x_2|)}\tau(N(x_3))N([x_1,N(x_2)]).
\end{eqnarray*}
\begin{eqnarray*}
N^2([N(x_1),x_2,x_3]_\tau)&=&\tau(N(x_1))N^2([x_2,x_3])-(-1)^{|x_1||x_2|}\tau(x_2)N^2([N(x_1),x_3])\\
&+&(-1)^{|x_3|(|x_1|+|x_2|)}\tau(x_3)N^2([N(x_1),x_2]).
\end{eqnarray*}
\begin{eqnarray*}
N^2([x_1,N(x_2),x_3]_\tau)&=&\tau(x_1)N^2([N(x_2),x_3])-(-1)^{|x_1||x_2|}\tau(N(x_2))N^2([x_1,x_3])\\
&+&(-1)^{|x_3|(|x_1|+|x_2|)}\tau(x_3)N^2([x_1,N(x_2)]).
\end{eqnarray*}
\begin{eqnarray*}
N^2([x_1,x_2,N(x_3)]_\tau)&=&\tau(x_1)N^2([x_2,N(x_3)])-(-1)^{|x_1||x_2|}\tau(x_2)N^2([x_1,N(x_3)])\\
&+&(-1)^{|x_3|(|x_1|+|x_2|)}\tau(N(x_3))N^2([x_1,x_2]).
\end{eqnarray*}
\begin{eqnarray*}
N^3([x_1,x_2,x_3]_\tau)&=&\tau(x_1)N^3([x_2,x_3])-(-1)^{|x_1||x_2|}\tau(x_2)N^3([x_1,x_3])\\
&+&(-1)^{|x_3|(|x_1|+|x_2|)}\tau(x_3)N^3([x_1,x_2]).
\end{eqnarray*}
We suppose that:\\
\begin{eqnarray*}
\small{ \mathcal{C}_\tau(x_1,x_2,x_3)}&=&N([N(x_1),N(x_2),x_3]_\tau)+N([N(x_1),x_2,N(x_3)]_\tau)+N([x_1,N(x_2),N(x_3)]_\tau)\\
&-&N^2([N(x_1),x_2,x_3]_\tau)-N^2([x_1,N(x_2),x_3]_\tau)-N^2([x_1,x_2,N(x_3)]_\tau)\\&+&N^3([x_1,x_2,x_3]_\tau),
\end{eqnarray*}
It  is easy to see that
\begin{eqnarray*}
[N(x_1),N(x_2),N(x_3)]_\tau-\mathcal{C}_\tau(x_1,x_2,x_3)&=&\tau(x_1)N\big([N(x_2),N(x_3)]-N([N(x_2),x_3])\\
&-&N([x_2,N(x_3)])+N^2([x_2,x_3])\big)\\
&-&(-1)^{|x_1||x_2|}\tau(x_2)N\big([N(x_1),N(x_3)]-N([N(x_1),x_3])\\
&-&N([x_1,N(x_3)])+N^2([x_1,x_3])\big)\\
&+&(-1)^{|x_3|(|x_1|+|x_2|)}\tau(x_3)N\big([N(x_1),N(x_2)]\\&-&N([N(x_1),x_2])
-N([x_1,N(x_2)])\\&+&N^2([x_1,x_2])\big)\\&=&0,
\end{eqnarray*}
which gives that $N$ is a Nijenhuis operator on $(\mathfrak{g},[\cdot,\cdot,\cdot]_\tau,\alpha,\beta)$.
\end{proof}
\begin{prop}
Let $N$ be a Nijenhuis operator on a $3$-BiHom-Lie superalgebras  $(\mathfrak{g},[\cdot,\cdot,\cdot],\alpha,\beta)$ and $R$ be a Rota-Baxter operator on a $3$-BiHom-Lie superalgebras  $(\mathfrak{g},[\cdot,\cdot,\cdot],\alpha,\beta)$ such that $R\circ N=N\circ R$ then $N$ be a Nijenhuis operator on a $3$-BiHom-Lie superalgebras  $(\mathfrak{g},[\cdot,\cdot,\cdot]_R,\alpha,\beta)$.

\end{prop}
\begin{proof} For all $x_1,x_2,x_3\in\mathcal{H}(\mathfrak{g})$, we have
\begin{eqnarray*}
[N(x_1),N(x_2),N(x_3)]_R&=&\displaystyle\sum_{\emptyset\neq I\subseteq[3]}\lambda^{|I|-1}[\widetilde{R}(N(x_1)),\widetilde{R}(N(x_2)),\widetilde{R}(N(x_3))]\\
&=& \displaystyle\sum_{\emptyset\neq I\subseteq[3]}\lambda^{|I|-1}[N(\widetilde{R}(x_1)),N(\widetilde{R}(x_2)),N(\widetilde{R}(x_3))]\\
&=&\displaystyle\sum_{\emptyset\neq I\subseteq[3]}\lambda^{|I|-1}\displaystyle\sum_{\emptyset\neq J\subseteq[3]}(-1)^{|J|-1}N^{|J|}[\widetilde{N}(\widetilde{R}(x_1)),\widetilde{N}(\widetilde{R}(x_2)),\widetilde{N}(\widetilde{R}(x_3))]\\
&=&\displaystyle\sum_{\emptyset\neq J\subseteq[3]}(-1)^{|J|-1}N^{|J|}\big(\displaystyle\sum_{\emptyset\neq I\subseteq[3]}\lambda^{|I|-1}[\widetilde{R}(\widetilde{N}(x_1)),\widetilde{R}(\widetilde{N}(x_2)),\widetilde{R}(\widetilde{N}(x_3))]\big)\\
&=&\displaystyle\sum_{\emptyset\neq J\subseteq[3]}(-1)^{|J|-1}N^{|J|}[\widetilde{N}(x_1),\widetilde{N}(x_2),\widetilde{N}(x_3)]_R,
\end{eqnarray*}
then $N$ is a Nijenhuis operator on a $3$-BiHom-Lie superalgebras  $(\mathfrak{g},[\cdot,\cdot,\cdot]_R,\alpha,\beta)$.
\end{proof}
\begin{prop}
Let $(\mathfrak{g},[\cdot,\cdot,\cdot],\alpha,\beta)$ be a $3$-Bihom-Lie superalgebra. If an even endomorphism $N$ is a derivation, then $N$ is a Nijenhuis operator if and only if  $N$ is a Rota-Baxter operator  of weight $0$ on $\mathfrak{g}$.
\end{prop}

\end{document}